\documentclass[a4paper]{amsart}
\usepackage{amsmath,amsthm,amssymb}
\usepackage[dvipdfmx]{graphicx}
\usepackage{mathrsfs}
\usepackage[dvipdfmx, usenames]{color}

\newtheorem{thm}{Theorem}[section]
\newtheorem{lem}[thm]{Lemma}
\newtheorem{cor}[thm]{Corollary}
\newtheorem{prop}[thm]{Proposition}

\theoremstyle{definition}

\newtheorem{defn}[thm]{Definition}
\newtheorem{assump}[thm]{Assumption}

\newtheorem{quest}[thm]{Question}

\theoremstyle{remark}

\newtheorem{rem}[thm]{Remark}        % \renewcommand{\therem}{}
\numberwithin{equation}{section}
\def\Xint#1{\mathchoice
{\XXint\displaystyle\textstyle{#1}}%
{\XXint\textstyle\scriptstyle{#1}}%
{\XXint\scriptstyle\scriptscriptstyle{#1}}%
{\XXint\scriptscriptstyle\scriptscriptstyle{#1}}%
\!\int}
\def\XXint#1#2#3{{\setbox0=\hbox{$#1{#2#3}{\int}$}
\vcenter{\hbox{$#2#3$}}\kern-.5\wd0}}

\def\dashint{\Xint-}
\newcommand{\Pro}{\mathcal{P}}
\newcommand{\R}{\mathbb{R}}

\newcommand{\Cpl}{\mathsf{Cpl}}
\newcommand{\di}{\mathrm{Diam}}
\newcommand{\supp}{\mathrm{supp}\,}
\newcommand{\WE}{\underline{\mathcal{WE}}}
\newcommand{\RE}{\mathcal{R}}
\newcommand{\LIP}{{\bf \mathrm{LIP}}}
\newcommand{\Ch}{\mathsf{Ch}}

\newcommand{\RCD}{\mathsf{RCD}}
\newcommand{\CD}{\mathsf{CD}}
\newcommand{\Geo}{\mathsf{Geo}}
\newcommand{\TestF}{\mathrm{TestF}}
\newcommand{\Tan}{\mathsf{Tan}}
\newcommand{\Dim}{\mathsf{dim}\,}

\newcommand{\calD}{\mathcal{D}}

\newcommand{\calH}{\mathcal{H}}

\newcommand{\calL}{\mathcal{L}}
\newcommand{\calM}{\mathcal{M}}

\newcommand{\calR}{\mathcal{R}}

\newcommand{\calX}{\mathcal{X}}

\newcommand{\bbN}{\mathbb{N}}

\newcommand{\bbS}{\mathbb{S}}

\newcommand{\frakn}{\mathfrak{n}}

\usepackage[abbrev,nobysame]{amsrefs}
\makeatletter
\def\@makefnmark{%
\leavevmode
\raise.9ex\hbox{\check@mathfonts
\fontsize\sf@size\z@\normalfont%
\@thefnmark}%
}
\makeatother
\title[positivity of regular sets]{A sufficient condition to a regular set of positive measure on $\RCD$ spaces}
\author{Yu Kitabeppu}
\address[Yu Kitabeppu]{Kumamoto University}
\email{ybeppu@kumamoto-u.ac.jp}
\begin{document}
\maketitle
 \begin{abstract}
  In this paper, we study regular sets in metric measure spaces with bounded Ricci curvature. We prove that the existence of a point in the regular set of the highest dimension implies the positivity of the measure of such regular set. Also we define the dimension of metric measure spaces and prove the lower semicontinuity of that under the Gromov-Hausdorff convergence.  
 \end{abstract}
 %%%%%%%%%%%%%%%%%%%%%%%%%%%%%%%%%%%%%%%%%%%%%%%%%%%%%%%%%%%%%%%%%%%%%%%%%%%%%%%%%%%%%%%%%%%%%%%%%%%%%%%%%%%%%%%%%%%%%%%%%%%%%%%%%%%%%%%%%%%%%%%%%%%%%%%%
 %
 %  Section : Introduction
 %
 %%%%%%%%%%%%%%%%%%%%%%%%%%%%%%%%%%%%%%%%%%%%%%%%%%%%%%%%%%%%%%%%%%%%%%%%%%%%%%%%%%%%%%%%%%%%%%%%%%%%%%%%%%%%%%%%%%%%%%%%%%%%%%%%%%%%%%%%%%%%%%%%%%%%%%%%
\section{Introduction}
In the series of papers \cite{CC1,CC2,CC3} by Cheeger and Colding, they investigate much properties of Ricci limit spaces. Especially, the study of the infinitesimal structure on such spaces is pretty important to understand the geometry of that. On a non-collapsing Ricci limit space $(Y,d,\nu)$, $\nu$-almost every point has unique tangent cone that is isometric to $N$-dimensional Euclidean space when the sequence of Rimannian manifolds approximating $Y$ are of $N$-dimension \cite{CC1}. It is also known that the limit measure $\nu$ is the $N$-dimensional Hausdorff measure multiplied by a constant. For collapsing Ricci limit spaces, the uniqueness of the dimension of tangent cones at almost every point had been an open problem. However Colding and Naber give the affirmative answer to the problem, that is, there exists a unique $k$ less than $N$ so that $\nu(Y\setminus \calR_k)=0$, where $\calR_k$ is the set of all points whose tangent cone is unique and isometric to $\R^k$(see \cite{CNholder}). Combining the results of Cheeger-Colding and of Colding-Naber leads the mutually absolutely continuity between the limit measure $\nu$ and the $k$-dimensional Hausdorff measure $\nu$-almost every point on $\calR_k$. 

On the other hand, $\RCD$ spaces are one of another generalization of Riemannian manifolds with Ricci curvature bounded from below. $\RCD$ spaces are defined as a kind of convexity of the functional on the Wasserstein spaces on those(see Section \ref{sec:pre}). In that sense, $\RCD$ spaces are defined by the intrinsic way while Ricci limit spaces are defined by the extrinsic way. It is known that all Ricci limit spaces are $\RCD$ spaces by the stability of $\RCD$ condition under the Gromov-Hausdorff convergence. Unfortunately, no one knows whether generic $\RCD$ spaces can be approximated by a family of Riemannian manifolds with uniformly bounded Ricci curvature or not. Hence the study on the geometry and analysis of $\RCD$ spaces are difficult because we are not able to use mathematical techniques developed on Riemannian manifolds. 
%Since Ricci limit spaces are approximated by manifolds in the Gromov-Hausdorff sense, geometry of that can be treated rather easier than $\RCD$ spaces. 
 However by using techniques from the study of optimal transportation problem on metric measure spaces instead of differential geometric ones, many geometric and analytic properties on $\RCD$ spaces are discovered, some of them are new even for Riemannian manifolds. On the other hand, the following fundamental question is still open, though the Ricci limit case is already answered positively; 
\begin{quest}\label{quest:main}
 Let $(X,d,m)$ be an $\RCD^*(K,N)$ space for $K\in\R$ and $N\in(1,\infty)$. Is there an integer $l$ with $1\leq l\leq[N]$ such that $m$-almost points $x\in X$ has unique tangent cone isomorphic to $\R^l$? 
\end{quest}
Mondino and Naber prove that $m\left(X\setminus \cup_{1\leq i \leq[N]}\calR_i\right)=0$ \cite{MN}. Question \ref{quest:main} are able to be reformulated that whether there is an integer $l$ such that $m(X\setminus \calR_l)=0$ or not. As mentioned above, this problem is completely solved in Ricci limit cases. Adding some conditions, we have partial positive answers for that problems \cite{KBishop,KL}. One of the main theorem in the present paper holds without any additional assumptions. 
%%%%%%%%%%
\begin{thm}[Proposition \ref{prop:kfunctions} and Proposition \ref{prop:lk}]\label{thm:posi}
 Let $(X,d,m)$ be an $\RCD^*(K,N)$ space for $k\in\R$ and $N\in(1,\infty)$. Assume $\calR_k\neq\emptyset$ for $1\leq k\leq[N]$. Then $m\left(\cup_{k\leq i\leq[N]}\calR_i\right)>0$. In particular, $m(\calR_{[N]})>0$ if $\calR_{[N]}\neq\emptyset$. 
\end{thm} 
%%%%%%%%%%
\begin{rem}
 In \cite{KL}, the author and Lakzian proved that $m(X\setminus \calR_1)=0$ provided $\calR_1\neq\emptyset$. Thus we only need to consider the case when $\calR_k\neq\emptyset$ for $k\geq 2$.  
\end{rem}
%%%%%%%%%% 
Let $k$ be the largest number such that $\calR_k\neq\emptyset$. It follows from Theorem \ref{thm:posi} that $m(\calR_k)>0$. If $m(X\setminus \calR_k)=0$, then Question \ref{quest:main} has the affirmative answer. But still, there exists a possibility that $m(\calR_i)>0$ for some $i$ less than $k$. Hence Theorem \ref{thm:posi} is far from the goal of our Question. However it must be a first step to that. 

By Theorem \ref{thm:posi}, we are able to define the dimension of $\RCD$ spaces. 
%%%%%%%%%%
\begin{defn}[Definition \ref{def:dimmm}]
 Let $(X,d,m)$ be an $\RCD^*(K,N)$ space for $K\in\R$, $N\in(1,\infty)$. The dimension of $(X,d,m)$ is defined as the largest number $k$ so that $\calR_k\neq\emptyset$ and $\calR_l=\emptyset$ for any $l>k$. $\Dim(X,d,m)$ denotes the dimension of $(X,d,m)$. Equivalently, the dimension of $(X,d,m)$ is also the largest number $k$ such that $m(\calR_k)>0$. 
\end{defn} 
%%%%%%%%%%
The concept of dimension here coincides with that introduced by Colding and Naber for Ricci limit spaces \cite{CNholder}. See \cite{KaLi}, \cite{HRicLp} for the proof. Also the \emph{analytic dimension} defined by Han \cite{HanRictensor} coincides(see Remark \ref{rem:analydim}). 
Another main theorem is as follows. 
\begin{thm}[Theorem \ref{thm:lscdimmm}]
 Let $(X_n,d_n,m_n,x_n)$ be a sequence of pointed $\RCD^*(K,N)$ spaces for $K\in\R$, $N\in(1,\infty)$ and converging to $(X_{\infty},d_{\infty},m_{\infty},x_{\infty})$. Then 
 \begin{align}
  \Dim(X_{\infty},d_{\infty},m_{\infty})\leq \liminf_{n\rightarrow\infty}\Dim(X_n,d_n,m_n).\notag
 \end{align}
\end{thm}
%%%%%%%%%%
\begin{rem}
 The lower semicontinuity of the dimension for tangent cones was already proven for Ricci limit case (\cite{KaLi, HRicLp}). 
\end{rem}
%%%%%%%%%%

 %%%%%%%%%%%%%%%%%%%%%%%%%%%%%%%%%%%%%%%%%%%%%%%%%%%%%%%%%%%%%%%%%%%%%%%%%%%%%%%%%%%%%%%%%%%%%%%%%%%%%%
 %
 %  Preliminaries
 %
 %%%%%%%%%%%%%%%%%%%%%%%%%%%%%%%%%%%%%%%%%%%%%%%%%%%%%%%%%%%%%%%%%%%%%%%%%%%%%%%%%%%%%%%%%%%%%%%%%%%%%
\section{Preliminaries}\label{sec:pre}
A triplet $(X,d,m)$ consisting of a complete separable metric space $(X,d)$ and a locally finite positive Borel measure $m$ on $X$ is called a \emph{metric measure space}. Two metric measure spaces $(X,d,m)$ and $(Y,r,\nu)$ are \emph{isomorphic} if there exists an isometry $f:\supp m\rightarrow\supp \nu$ with $f_*m=\nu$. 
A continuous curve $\gamma:[0,1]\rightarrow X$ is \emph{absolutely continuous} if there exists an $L^1(0,1)$ function $g$ such that 
\begin{align}
  d(\gamma_s,\gamma_t)\leq \int_s^tg(r)\,dr,\qquad\text{for any }s\leq t\in[0,1].\notag
\end{align}
For a continuous curve $\gamma:[0,1]\rightarrow X$, the metric derivative $\vert\dot\gamma\vert$ is defined by 
\begin{align}
 \vert\dot\gamma_t\vert:=\lim_{s\rightarrow t}\frac{d(\gamma_t,\gamma_s)}{\vert s-t\vert}\notag
\end{align}
as long as the right-hand side makes sense.  
It is known that every absolutely continuous curve has the metric derivative for almost every point \cite{AGSbook}. We call an absolutely continuous curve $\gamma:[0,1]\rightarrow X$ a \emph{geodesic} if $\vert\dot\gamma_t\vert=d(\gamma_0,\gamma_1)$ for almost every $t\in [0,1]$. A metric space $(X,d)$ is called a \emph{geodesic space} if for any two points, there exists a geodesic connecting them. 

We denote the set of all Lipschitz functions on $X$ by $\LIP(X)$. For $f\in\LIP(X)$, the local Lipschitz constant at $x$, $\vert \nabla f\vert(x)$, is defined as
\begin{align} 
 \vert \nabla f\vert(x):=\limsup_{y\rightarrow x}\frac{\vert f(x)-f(y)\vert}{d(x,y)}\notag
\end{align} 
if $x$ is not isolated, otherwise $\vert \nabla f\vert(x)=\infty$. For $f\in L^2(X,m)$, we define the \emph{Cheeger energy} $\Ch(f)$ as 
\begin{align}
 \Ch(f):=\frac{1}{2}\inf\left\{\liminf_{n\rightarrow \infty}\int_X\vert \nabla f_n\vert^2\,dm\;;\;f_n\in\LIP(X),\;f_n\rightarrow f\text{ in }L^2(X,m)\right\}.\notag
\end{align}
Set $D(\Ch):=\{f\in L^2(X,m)\;;\;\Ch(f)<\infty\}$. It is known that for any $f\in D(\Ch)$, there exists an $L^2$-function $\vert \nabla f\vert_w$ such that $\Ch(f)=(1/2)\int\vert \nabla f\vert_w^2\,dm$, which is called a \emph{minimal weak upper gradient}. For simplicity, we denote the minimal weak upper gradient of $f$ just by $\vert \nabla f\vert$. We define the \emph{Sobolev space} $W^{1,2}(X,d,m):=L^2(X,m)\cap D(\Ch)$ equipped with the norm $\Vert f\Vert_{1,2}^2:=\Vert f\Vert_2^2+2\Ch(f)$. It is known that $W^{1,2}$ is a Banach space. We say that $(X,d,m)$ is \emph{infinitesimally Hilbertian} if $W^{1,2}$ is a Hilbert space. 

We denote the set of all Borel probability measures on $X$ by $\Pro(X)$. We define $\Pro_2(X)$ as the set of all Borel probability measures with finite second moment, that is, $\mu\in\Pro_2(X)$ if and only if $\mu\in\Pro(X)$ and there exists a point $o\in X$ such that $\int_Xd(x,o)^2\,d\mu(x)<\infty$. We call a measure $q\in\Pro(X\times X)$ a \emph{coupling} between $\mu$ and $\nu$ if $(p_1)_*q=\mu$ and $(p_2)_*q=\nu$, where $p_i:X\times X\rightarrow X$ are natural projections for $i=1,2$. For two probability measures $\mu_0,\mu_1\in\Pro_2(X)$, we define the \emph{$L^2$-Wasserstein distance} between $\mu_0$ and $\mu_1$ as 
\begin{align}
 W_2(\mu_0,\mu_1):=\inf\left\{\int_{X\times X}d(x,y)^2\,dq(x,y)\;;\;q\in\Cpl(\mu_0,\mu_1)\right\}^{1/2},\notag
\end{align}
where $\Cpl(\mu_0,\mu_1)$ is the set of all couplings of $\mu_0$ and $\mu_1$. The pair $(\Pro_2(X),W_2)$ is called the \emph{$L^2$-Wasserstein space}, which is a complete separable geodesic metric space if so is $(X,d)$. We explain how geodesics in $X$ relates to those in $\Pro_2(X)$. We denote the space of all geodesics in $X$ by $\Geo(X)$, equipped with the sup distance. Define the \emph{evaluation map} $e_t:\Geo(X)\rightarrow X$ for $t\in [0,1]$ by $e_t(\gamma)=\gamma_t$. Let $(\mu_t)_t\in \Geo(\Pro_2(X))$ be a geodesic connecting $\mu_0,\mu_1$ in $\Pro_2(X)$. Then there exists a probability measure $\pi\in\Pro(\Geo(X))$ such that $(e_t)_*\pi=\mu_t$, by which we say that the geodesic $(\mu_t)_t$ can be lifted to $\pi$. 
%%%%%%%%%%%%%%%%%%%%%%%%%%%%%%%%%%%%%%%%%%%%%%%%%%%%%%%%%%%%%%%%%%%%%%%%%%%%%%%%%%%%%%%%%%%%%%%%%%%%%%%%%%%%%%%%%%%%%%%%%%%%%%%%%%%%%%%%%%%%%%%%%%%%%%%%%
\subsection{The curvature-dimension condition}
 For given $K\in\R$ and $N\in(1,\infty)$, we define the \emph{distortion coefficients}, $\sigma^{(t)}_{K,N}$ for $t\in[0,1]$, by 
 \begin{align}
  \sigma^{(t)}_{K,N}(\theta):=\begin{cases}
                                            \infty&\text{if }K\theta^2\geq N\pi^2,\\
                                            \frac{\sin(t\theta\sqrt{K/N})}{\sin(\theta\sqrt{K/N})}&\text{if }0<K\theta^2<N\pi^2,\\
                                            t&\text{if }K\theta^2=0,\\
                                            \frac{\sinh(t\theta\sqrt{-K/N})}{\sinh(\theta\sqrt{-K/N})}&\text{if }K\theta^2<0.
                                           \end{cases}\notag
 \end{align}
 Let $(Y,d)$ be a geodesic metric space and $f:Y\rightarrow\R\cup\{\pm\infty\}$ a function on $Y$. 
 %%%%%%%%%%%%
 %%%%%%%%%%%% (K,N)-convexity
 \begin{defn}[\cite{EKS}]
  A function $f:Y\rightarrow \R\cup\{\pm\infty\}$ is said to be \emph{$(K,N)$-convex} for $K\in\R$ and $N\in(1,\infty)$ if for any two points $y_0,y_1\in Y$, there exists a geodesic $(y_t)_t$ connecting them such that 
  \begin{align}
   &\exp\left(-\frac{1}{N}f(y_t)\right)\notag\\
   &\geq \sigma^{(1-t)}_{K,N}(d(y_0,y_1))\exp\left(-\frac{1}{N}f(y_0)\right)+\sigma^{(t)}_{K,N}(d(y_0,y_1))\exp\left(-\frac{1}{N}f(y_1)\right)\notag
  \end{align} 
  holds for any $t\in[0,1]$. 
 \end{defn} 
 %%%%%%%%%%%%
 Let $(X,d,m)$ be a geodesic metric measure space. Consider $\mu=\rho m\ll m$ a probability measure that is absolutely continuous with respect to $m$ and its Radon-Nikodym derivative being $\rho$. We define the relative entropy functional $\mathrm{Ent}_m$ by 
 \begin{align}
  \mathrm{Ent}_m(\mu):=\int_{\{\rho>0\}}\rho\log\rho\,dm,\notag
 \end{align}
 whenever $(\rho\log\rho)_+$ is integrable, otherwise we define $\mathrm{Ent}_m(\mu)=\infty$. 
 %%%%%%%%%%%% CD^e(K,N)
 \begin{defn}[\cite{EKS}, cf. \cite{AMSnon}]
  Let $(X,d,m)$ be a geodesic metric measure space. We say that $(X,d,m)$ satisfies the \emph{entropic curvature-dimension condition} $\CD^e(K,N)$ for $K\in\R$ and $N\in(1,\infty)$ if the relative entropy functional $\mathrm{Ent}_m$ is $(K,N)$-convex. Moreover if $(X,d,m)$ is infinitesimally Hilbertian, $(X,d,m)$ is called an $\RCD^*(K,N)$ space. 
 \end{defn}
 %%%%%%%%%%%%
 Under the infinitesimal Hilbertianity condition, $\CD^e(K,N)$ is equivalent to $\CD^*(K,N)$ \cite{EKS}. 
 %%%%%%%%%%%%%%%%%%%%%%%%%%%%%%%%%%%%%%%%%%%%%%%%%%%%%%%%%%%%%%%%%%%%%%%%%%%%%%%%%%%%%%%%%%%%%%%%%%%%%%%%%%%%%%%%%%%%%%%%%%%%%%%%%%%%%%%%%%%%%%%%%%%%%%%%%
 \subsection{Tangent cones and regular sets on $\RCD$ spaces}
 Let $(X,d,m)$ be a metric measure space. Take a point $x_0\in \supp m$ and fix it. We call a quadruple $(X,d,m,x_0)$ a \emph{pointed metric measure space}. We say that a pointed metric measure space $(X,d,m,x_0)$ is \emph{normalized} if 
 \begin{align}
  \int_{B_1(x_0)}1-d(x_0,\cdot)\,dm=1.\label{eq:normal}
 \end{align}
  For $r\in(0,1)$, define $d_r:=d/r$ and 
 \begin{align}
  m^x_r:=\left(\int_{B_r(x)}1-d_r(x,\cdot)\,dm\right)^{-1}m.\notag
 \end{align}
 Note that the pointed metric measure space $(X,d_r,m^x_r,x)$ is normalized. 
 
 Let $C(\cdot):[0,\infty)\rightarrow [1,\infty)$ be a nondecreasing function. Define $\mathcal{M}_{C(\cdot)}$ the family of pointed metric measure spaces $(X,d,m,\bar x)$ that satisfy 
 \begin{align}
  m(B_{2r}(x))\leq C(R)m(B_r(x))\notag
 \end{align}
 for any $x\in \mathrm{supp}\,m$, and any $0<r\leq R<\infty$. Gigli, Mondino, and Savar\'e have proven that there exists a distance function $\mathcal{D}_{C(\cdot)}:\mathcal{M}_{C(\cdot)}\times \mathcal{M}_{C(\cdot)}\rightarrow [0,\infty]$, which induces the same topology as the Gromov-Hausdorff one on $\mathcal{M}_{C(\cdot)}$ (\cite{GMS}). It is known that every $\RCD^*(K,N)$ space for given $K\in\R$, $N\in (1,\infty)$ belongs to $\calM_{C(\cdot)}$ for a common function $C(\cdot):(0,\infty)\rightarrow [1,\infty)$ (see \cite{Stmms2,FS}), more precisely, they satisfy 
 \begin{align}
  \frac{m(B_R(x))}{m(B_r(x))}\leq \frac{\int_0^R\bbS_{K,N}^{N-1}(t)\,dt}{\int_0^r\bbS_{K,N}^{N-1}(t)\,dt}\notag
 \end{align}
 for any $x\in\supp m$, $0<r\leq R$, where 
 \begin{align}
  \bbS_{K,N}(t):=\begin{cases}
                           \sin\left(t\sqrt{\frac{K}{N-1}}\right)&\text{if }K>0,\\
                           t&\text{if }K=0,\\
                           \sinh\left(t\sqrt{\frac{-K}{N-1}}\right)&\text{if }K<0. 
                          \end{cases}\notag
 \end{align}
 Let $(X,d,m)$ be an $\RCD^*(K,N)$ space for $K\in\R$ and $N\in(1,\infty)$. By a simple calculation, we have $(X,d_r,m_r^x)$ for some $x\in \supp m$ being an $\RCD^*(r^2K,N)$ space. Take a point $x\in\supp m$ and fix it. Consider the family of normalized metric measure spaces $\{(X,d_r,m^x_r,x)\}_{r\in(0,1)}$. The following is one of a generalization of Gromov's compactness theorem. 
 %%%%%%%%%% Precompactness
 \begin{thm}[\cite{GMR}]
  The family of normalized metric measure spaces $\{(X,d_r,m^x_r)\}_{r\in(0,1)}$ is sequentially compact with respect to the pointed measured Gromov-Hausdorff topology. Moreover every limit space $(X,d_{r_n},m^x_{r_n},x)\rightarrow (Y,d_Y,m_Y,y)$ is a normalized $\RCD^*(0,N)$ space for a non-increasing sequence $\{r_n\}_n$ with $r_n\rightarrow 0$. 
 \end{thm} 
 %%%%%%%%%%
 We define the \emph{tangent cones} at a point $x\in\supp m$ by 
 \begin{align}
  \Tan(X,d,m,x):=\left\{(Y,d_Y,m_Y,y)\;;\;(X,d_{r_n},m^x_{r_n},x)\rightarrow(Y,d_Y,m_Y,y)\right\},\notag
 \end{align}
 where $\{r_n\}_n$ is a non-increasing sequence converging to 0. For simplicity, we just denote by $\Tan(X,x)$ instead of $\Tan(X,d,m,x)$ if there is no confusion.  
 %%%%%%%%%% Regular sets
 \begin{defn}\label{def:regularsets}
  Let $(X,d,m)$ be an $\RCD^*(K,N)$ space for $K\in\R$ and $N\in(1,\infty)$. We call a point $x\in\supp m$ a $k$-\emph{regular point} if $\Tan(X,x)=\{(\R^k,d_E,\underline{\mathcal{L}}^k,0)\}$, where $\underline{\mathcal{L}}^k$ is the normalized Lebesgue measure at $0$. We denote the set of $k$-regular points by $\RE_k$. 
 \end{defn}
 %%%%%%%%%% 
 Mondino and Naber proved the following \cite{MN}. 
 %%%%%%%%%% Regularity result
 \begin{thm}[\cite{MN}*{Corollary 1.2}]\label{thm:MN}
  Let $(X,d,m)$ be an $\RCD^*(K,N)$ space for $K\in\R$ and $N\in(1,\infty)$. Then 
  \begin{align}
   m\left(X\setminus \bigcup_{1\leq k\leq N}\RE_k\right)=0.\notag
  \end{align}
 \end{thm}
 %%%%%%%%%%
 Note that even though by using Theorem \ref{thm:MN}, we do not know the uniqueness of regular sets in general. For Ricci limit spaces, we know $m(X\setminus \calR_l)=0$ for some $1\leq l\leq N$ (\cite{CNholder}), and in the following restricted case, the uniqueness of regular sets is known; 
 \begin{enumerate}
  \item $\calR_1\neq\emptyset$ (\cite{KL}),
  \item A Bishop type inequality holds (\cite{KBishop}). 
 \end{enumerate}
 %%%%%%%%%%%%%%%%%%%%%%%%%%%%%%%%%%%%%%%%%%%%%%%%%%%%%%%%%%%%%%%%%%%%%%%%%%%%%%%%%%%%%%%%%%%%%%%%%%%%%%%%%%%%%%%%%%%%%%%%%%%%%%%%%%%%%%%%%%%%%%%%%%%%%%%%
 \subsection{Smoothing effects of the heat flows and the modified heat flows on $\RCD$ spaces}
 Let $(X,d,m)$ be an $\RCD^*(K,N)$ space for $K\in\R$ and $N\in(1,\infty)$. By the infinitesimal Hilbertianity, $\Ch$ is actually a strongly local Dirichlet form (\cite{AGSRiem}). Let $\Delta$ be the generator of $\Ch$, called the \emph{Laplacian} and $\{h_t\}_{t>0}$ the associated semigroup, called \emph{the heat flow}. Thus for any $L^2$-function $f$, we have 
 \begin{align}
  \frac{d}{dt}h_tf=\Delta h_tf\notag
 \end{align}
 in $L^2$-sense. It is known that $\RCD(K,\infty)$ condition is equivalent to the \emph{Bakry-\'Emery curvature-dimension condition $\mathsf{BE}(K,\infty)$}, which requires 
 \begin{align}
  \vert\nabla h_tf\vert^2\leq e^{-2Kt}h_t(\vert \nabla f\vert^2)\quad m\text{-a.e., }\notag
 \end{align}
 for any $t>0$ and any $f\in W^{1,2}(m)$ (see \cite{AGSRiem}). 
 We define the following set of test functions; 
 \begin{align}
  \TestF(X):=\left\{f\in\calD(\Delta)\cap L^{\infty}(m)\;;\;\vert \nabla f\vert\in L^{\infty}(m)\text{ and }\Delta f\in W^{1,2}(m)\right\}.\notag
 \end{align}
 By $\mathsf{BE}(K,\infty)$, $h_tf\in \TestF(X)$ for any $t>0$ and any $f\in L^2\cap L^{\infty}(m)$. In order to obtain more regular functions, we define \emph{the modified heat flow} $\hat h_t$ as 
 \begin{align}
  \hat h_tf:=\frac{1}{t}\int_0^{\infty} h_sf\phi(st^{-1})\,ds,\notag
 \end{align}
 where $\phi\in C^{\infty}_c(0,1)$ with $\int_0^1\phi(s)\,ds=1$ is a given nonnegative function. It is known that for a Lipschitz function $f$, $\hat h_tf$ satisfies $\hat h_tf\in \TestF(X)$, $\Delta \hat h_t f\in L^{\infty}(m)$, and $\hat h_t f\rightarrow f$ in $W^{1,2}$ (cf. \cite{Gnon}). The latter property plays a crucial role in the proof of the main theorem. 
 %%%%%%%%%%%%%%%%%%%%%%%%%%%%%%%%%%%%%%%%%%%%%%%%%%%
 \subsubsection{Smooth cut-off functions}\label{subsection:cutoff}
 In \cite{MN}, Mondino and Naber define a smooth cut-off function on an $\RCD$ space as follows. Let $(X,d,m)$ be an $\RCD^*(K,N)$ space for $K\in \R$, $N\in(1,\infty)$. For every $x\in X$, $R>0$, and $0<r<R$, there exists a Lipschitz function $\varphi^x_r:X\rightarrow \R$ such that 
 \begin{enumerate}
  \item $0\leq \varphi^x_r\leq 1$ on $X$, $\varphi^x_r\equiv 1$ on $B_r(x)$ and $\supp \varphi^x_r\subset B_{2r}(x)$,
  \item $r^2\vert \Delta \varphi^x_r\vert+r\vert \nabla \varphi^x_r\vert\leq C(K,N,R)$, where the constant $C$ depends only on $K$, $N$, and $R$. 
 \end{enumerate} 
 Note that on the rescaled space $(X,d_r,m^x_r)$, the smooth cut-off function $\varphi^{d_r,x}_1:=\varphi^x_r$ satisfies $\vert \Delta^{d_r}\varphi^{d_r,x}_1\vert+\vert\nabla^{d_r}\varphi^{d_r,x}_1\vert\leq C$ (see Lemma \ref{lem:rescaleeffect}).
 %%%%%%%%%%%%%%%%%%%%%%%%%%%%%%%%%%%%%%%%%%%%%%%%%%%%%%%%%%%%%%%%%%%%%%%%%%%%%%%%%%%%%%%%%%%%%%%%%%%%%%%%%%%%%%%%%%%%%%%%%%%%%%%%%%%%%%%%%%%%%%%%%%%%%%%%%
 \subsection{Convergence results}
 Under the measured Gromov-Hausdorff convergence, the behavior of functions and their differentials are big issues. Gigli-Mondino-Savar\'e \cite{GMS}, Ambrosio-Stra-Trevisan \cite{AST}, and Ambrosio-Honda \cite{AH} study such behaviors. In this subsection, we just list their theorems we use later. For this purpose, we give some fundamental concepts of convergences of various objects. For a complete separable metric space $(X,d)$, we denote by $\calM_{loc}(X)$ the set of all Borel measures which are finite on every bounded sets. It says that a sequence $\mu_i\in\calM_{loc}(X)$ converges to $\mu\in\calM_{loc}(X)$ \emph{weakly} if 
 \begin{align}
  \lim_{i\rightarrow\infty}\int_X\phi\,d\mu_i=\int_X\phi\,d\mu\notag
 \end{align}
 holds for any $\phi\in C_{bs}(X)$, the set of all continuous functions with bounded supports. 
 ~\\
 \par In \cite{GMS}, the convergence of metric measure spaces are discussed in the fully general setting. In particular one of the main consequences in \cite{GMS} is the coincidence of concepts of \emph{the measured Gromov-Hausdorff convergence} and \emph{the measured Gromov convergence} under the uniform doubling condition. In this paper we also use the so called ``the extrinsic convergence concept" along \cite{GMS}, \cite{AH}, and \cite{AST}. 
 %%%%%%%%%%%%%%%%%%%%%%%%%%%%%%%%%%%%%%%%%%%%%%%%%%%
 \subsubsection{The pmG convergence}
 Let $(X_i,d_i,m_i,x_i)$ be a sequence of pointed metric measure spaces and $(X_{\infty},d_{\infty},m_{\infty},x_{\infty})$ a pointed metric measure space. We say that $(X_i,d_i,m_i,x_i)$ converges to $(X_{\infty},d_{\infty},m_{\infty},x_{\infty})$ in \emph{the pointed measured Gromov convergence $($pmG$)$ sense} if there exist a complete separable metric space $(X,d)$, isometric embeddings $\iota_{i}:X_i\rightarrow X$ for $i\in \bbN\cup\{\infty\}$ such that $\iota_i(x_i)\rightarrow\iota_{\infty}(x_{\infty})$ and $\frakn_i:={(\iota_i)}_*m_i$ converges to $\frakn_{\infty}:={(\iota_{\infty})}_*m_{\infty}$ weakly in $\calM_{loc}(X)$. Hereafter we identify $(X_i,d_i,m_i)$ and $(\iota_i(X_i), d_{\iota_i(X_i)},\frakn_i)$. Hence $(X_i,d_i,m_i,x_i)$ converges to $(X_{\infty},d_{\infty},m_{\infty},x_{\infty})$ is equivalent to $m_i$ converges to $m_{\infty}$ weakly in $\calM_{loc}(X)$ and $x_i\rightarrow x_{\infty}$. In this setting, we are able to consider the convergence of a sequence of functions on varying spaces. 
 %%%%%%%%%%%%%%%%%%%%%%%%%%%%%%%%%%%%%%%%%%%%%%%%%%%
 \subsubsection{Convergence of a sequence of functions on varying spaces}
 For $1<p<\infty$, we say that $f_i\in L^p(X_i,m_i)$ converges to $f_{\infty}\in L^p(X_{\infty},m_{\infty})$ weakly if $f_im_i\rightarrow f_{\infty}m_{\infty}$ weakly in $\calM_{loc}(X)$ and $\limsup_{i\rightarrow\infty}\Vert f_i\Vert_{L^p(m_i)}<\infty$. Moreover we say that $f_i\in L^p(m_i)$ converges to $f_{\infty}\in L^p(m_{\infty})$ \emph{strongly} if $f_i\rightarrow f_{\infty}$ weakly and $\limsup_{i\rightarrow\infty}\Vert f_i\Vert_{L^p(m_i)}\leq \Vert f_{\infty}\Vert_{L^p(m_{\infty})}$. By a usual argument, we also have 
 \begin{align}
  \Vert f_{\infty}\Vert_{L^2(m_{\infty})}\leq \liminf_{i\rightarrow\infty}\Vert f_i\Vert_{L^2(m_i)}\notag
 \end{align} 
 if $f_i$ $L^2$-weakly converges to $f_{\infty}$. We are able to define the $L^1$-strongly convergence. $f_i\in L^1(X,m_i)$ $L^1$-strongly converges to $f\in L^1(X,m_{\infty})$ if $\mathrm{sign}(f_i)\sqrt{\vert f_i\vert}$ $L^2$-strongly converges to $\mathrm{sign}(f)\sqrt{\vert f\vert}$.  
 %%%%%%%%%%%%%%%%%%%%%%%%%%%%%%%%%%%%%%%%%%%%%%%%%%%
 \subsubsection{Convergence of a sequence of functions in $W^{1,2}$}
 From now on, we always assume $(X,d,m_i)$, $i\in\bbN\cup\{\infty\}$ are $\RCD^*(K,N)$ spaces for $K\in\R$, $N\in(1,\infty)$. Note that the concepts of pmG convergence coincides with that of pmGH one in this setting. We denote by $\Ch^i$ the Cheeger energy on each $L^2(m_i)$. We say that $f_i\in W^{1,2}(m_i)$ converges to $f_{\infty}\in W^{1,2}(m_{\infty})$ \emph{weakly} if $f_i$ $L^2$-weakly converges to $f_{\infty}$ and $\sup_i\Ch^i(f_i)<\infty$. Moreover we say $f_i$ are \emph{$W^{1,2}$-strongly} converges to $f_{\infty}$ if $f_i$ $L^{2}$-strongly converges to $f_{\infty}$ and $\lim_{i\rightarrow\infty}\Ch^i(f_i)=\Ch^{\infty}(f_{\infty})$. We denote by $\Delta_i$ the Laplacian corresponding to $\Ch^i$, that is, the generator of that and by $\calD(\Delta_i)\subset W^{1,2}(m_i)$ its domain. 
 %%%%%%%%%%% Convergence of gradients
 \begin{thm}\label{thm:convgrad}
  Assume $(X,d,m_i,x_i)\xrightarrow{pmG} (X,d,m_{\infty},x_{\infty})$ and assume $f_i\in \calD(\Delta_i)$ $L^2$-strongly converges to $f$ and $\sup_i\Vert \Delta_if_i\Vert_{L^2(m_i)}<\infty$, then $f\in\calD(\Delta_{\infty})$, $\Delta_if_i$ $L^2$-weakly converges to $\Delta_{\infty} f$, and $f_i$ $W^{1,2}$-strongly converges to $f$. 
 \end{thm}
 %%%%%%%%%%%
 The following theorems play important roles.
 %%%%%%%%%%% 
 \begin{thm}[\cite{AH}*{Theorem 7.4}]\label{thm:precptSobolev}
  Assume $f_i\in W^{1,2}(m_i)$ satisfies 
  \begin{align}
   \sup_i\int_X\vert f_i\vert^2\,dm_i+\Ch^i(f_i)<\infty\notag
  \end{align}
  and for some $($hence all$)$ $\bar x\in X$ 
  \begin{align}
   \lim_{R\rightarrow\infty}\limsup_{i\rightarrow\infty}\int_{X\setminus B_R(\bar x)}\vert f_i\vert^2\,dm_i=0.\notag
  \end{align}
  Then $(f_i)$ has a $L^2$-strongly convergent subsequence to a function $f\in W^{1,2}(m_{\infty})$. 
 \end{thm}
 %%%%%%%%%%%
 %%%%%%%%%%%
 \begin{thm}[\cite{AH}*{Theorem 5.4, Theorem 5.6}]\label{thm:convinnerproduct}
  Assume that $(X_i,d_i,m_i,p_i)$ pmG converges to $(X_{\infty},d_{\infty},m_{\infty},p_{\infty})$ and assume $(X_i,d_i,m_i)$ are $\RCD^*(K,N)$ spaces for $K\in\R$, $N\in(1,\infty)$. Then 
  \begin{align}
   \lim_{i\rightarrow\infty}\int_{X_i}\Gamma^i(v_i,w_i)\,dm_i=\int_{X_{\infty}}\Gamma^{\infty}(v, w)\,dm_{\infty}
  \end{align}
  whenever $v_i$ strongly converges in $W^{1,2}$ to $v$ and $w_i$ weakly converges in $W^{1,2}$ to $w$. Moreover $\Gamma^i(v_i,w_i)$ $L^1$-strongly converges to $\Gamma^{\infty}(v,w)$ whenever $w_i$ also strongly converges in $W^{1,2}$ to $w$.  
 \end{thm}
%%%%%%%%%%%%%%%%%%%%%%%%%%%%%%%%%%%%%%%%%%%%%%%%%%%%%%%%%%%%%%%%%%%%%%%%%%%%%%%%%%%%%%%%%%%%%%%%%%%%%%%%%%%%%%%%%%%%%%%%%%%%%%%%%%%%%%%%%%%%%%%%%%%%%%%%%
%
%  The existence of a point in a regular set
%
%%%%%%%%%%%%%%%%%%%%%%%%%%%%%%%%%%%%%%%%%%%%%%%%%%%%%%%%%%%%%%%%%%%%%%%%%%%%%%%%%%%%%%%%%%%%%%%%%%%%%%%%%%%%%%%%%%%%%%%%%%%%%%%%%%%%%%%%%%%%%%%%%%%%%%%%%
\section{The existence of a point in a regular set}\label{sec:main}
Let $\Psi(\epsilon; a)$ denote a function depending on $a\in\R$ and tending 0 as $\epsilon\rightarrow 0$. The following lemma plays a key role in the proof of main results. 
%%%%%%%%%%%%
\begin{lem}[\cite{HRicLip}*{Lemma 3.1}]\label{lem:telescope}
 Let $Z$ be a proper geodesic space, $z\in Z$, and $\nu$ a Radon measure on $Z$. 
 For $s,\delta>0$, assume a nonnegative Borel function $F$ defined on $B_s(z)$ satisfies 
 \begin{align}
  \frac{1}{\nu(B_s(z))}\int_{B_s(z)}F\,d\nu\leq \delta\notag
 \end{align}
 and assume there exists a real number $\kappa\geq 1$ such that 
 \begin{align}
  0<\nu(B_{2t}(w))\leq 2^\kappa\nu(B_t(w))\notag
 \end{align}
 holds for any $w\in B_s(z)$, any $0<t\leq s$. 
 Then there exists a compact set $K\subset \overline{B}_{s/10^2}(z)$ with $\nu(K)/\nu(B_{s/10^2}(z))\geq 1-\Psi(\delta;\kappa)$ such that 
 \begin{align}
  \frac{1}{\nu(B_t(x))}\int_{B_t(x)}F\,d\nu\leq \Psi(\sqrt{\delta};\kappa),\notag
 \end{align}
 where $\Psi(x;a_1,\cdots,a_n)$ is a function tending to 0 as $x\rightarrow 0$ of order $x$ and depending only on $a_1,\cdots,a_n$. 
\end{lem}
%%%%%%%%%%%%
Recall the local structure theorem by Mondino and Naber \cite{MN}. Let $(X,d,m)$ be an $\RCD^*(K,N)$ space for $N\in(1,\infty)$. Then there exists a real number $\beta=\beta(N)>2$ such that the following theorem holds. 
%%%%%%%%%%%% Thm : Mondino-Naber's local structure theorem
 \begin{thm}[\cite{MN}*{Theorem 4.1}]\label{thm:MNexcess}
  Take a point $\bar x\in \calR_k$ and fix it. For any sufficiently small $\epsilon>0$, there exists a large number $\bar R=\bar R(\epsilon)\gg 1$ such that for any $R\geq \bar R$, there exists a small number $0<r=r(\bar x,\epsilon, R)\ll 1$ the following holds; there exist pairs of points $\{p_i,q_i\}_{i=1,\ldots,k}\subset B^{d_r}_{R^{\beta}}(\bar x)$ and $\{p_i+p_j\}_{1\leq i<j\leq k}\subset B^{d_r}_{2R^{\beta}}(\bar x)\setminus B^{d_r}_{R^{\beta}}(\bar x)$ such that 
  \begin{align}
   &\sum_{i=1}^k\dashint_{B^{d_r}_R(\bar x)}\vert \nabla e_{p_i,q_i}\vert^2\,dm^{\bar x}_r+\sup_{y\in B^{d_r}_R(\bar x)}e_{p_i,q_i}(y)\leq \epsilon\label{eq:derexcess}\\
   &\sum_{1\leq i<j\leq k}\dashint_{B^{d_r}_R(\bar x)}\left\vert \nabla \left(\frac{d_r^{p_i}+d_r^{p_j}}{\sqrt2}-d_r^{p_i+p_j}\right)\right\vert^2\,dm^{\bar x}_r+\sup_{B^{d_r}_R(\bar x)}\left\vert \frac{d_r^{p_i}+d_r^{p_j}}{\sqrt 2}-d_r^{p_i+p_j}\right\vert\leq \epsilon, \label{eq:derindep}
  \end{align}
  where $d_r^{p_i}(\cdot):=d_r(p_i,\cdot)$. 
 \end{thm}
%%%%%%%%%%%%
From now on, without loss of generality, we assume the following. 
%%%%%%%%%%%%
\begin{assump}
 Let $(X,d,m)$ be an $RCD^*(K,N)$ space and $\bar x\in \calR_k$. There exist pairs of points $\{p_i,q_i\}_{i=1,\ldots,k}\subset B_{R^{\beta}}(\bar x)$ and $\{p_i+p_j\}_{1\leq i<j\leq k}\subset B_{2R^{\beta}}(\bar x)\setminus B_{R^{\beta}}(\bar x)$ such that (\ref{eq:derexcess}) and (\ref{eq:derindep}) holds for sufficiently small $\epsilon>0$ and for replacing $d_r$ and $m^{\bar x}_r$ by just $d$ and $m$ respectively. 
\end{assump}
%%%%%%%%%%%%
It is known that $\vert \nabla d^{p}\vert\equiv 1$ $m$-a.e.(see \cite{GP}*{Proposition 3.1}). By using the calculation in \cite{MN,GP}, we have 
\begin{align}
 \vert\Gamma(d^{p_i},d^{p_j})\vert\leq 10\left\vert \nabla\left(\frac{d^{p_i}+d^{p_j}}{\sqrt 2}-d^{p_i+p_j}\right)\right\vert^2.\notag
\end{align}
Thus we obtain that 
\begin{align}
 \dashint_{B_R(\bar x)}\left\vert \Gamma(d^{p_i},d^{p_j})-\delta_{ij}\right\vert\,dm\leq 10\epsilon.\notag
\end{align}
Define a function $f_i:=d^{p_i}\varphi^{\bar x}_R$ such that $f_i\in W^{1,2}(X,d,m)$ with bounded support, where $\varphi^{\bar x}_R$ is a smooth cut-off function introduced in subsection \ref{subsection:cutoff}. Then we have 
\begin{align}
 \dashint_{B_R(\bar x)}\vert \Gamma(f_i,f_j)-\delta_{ij}\vert\,dm\leq 10\epsilon.\notag 
\end{align}
Let us consider smoothing functions $g^i:=\hat h_t f_i$, $i=1,\ldots, k$ such that 
\begin{align}
 \dashint_{B_R(\bar x)}\vert \Gamma(g^i,g^j)-\delta_{ij}\vert\,dm\leq 20\epsilon.\notag
\end{align} 
%We modify the functions $f_i$ again by using the modified heat flow 
%\begin{align}
% \tilde h_{\eta}f_i:=\frac{1}{\eta}\int_0^{\infty}h_sf_i\phi(s\eta^{-1})\,ds\;\text{for some given }\phi\in C^{\infty}_c(0,1)\text{ with }\int_0^1\phi(s)\,ds=1\notag
%\end{align}
%for $\eta>0$. It is known that $\tilde h_{\eta}f_i\rightarrow f_i$ in $W^{1,2}$ as $\eta\rightarrow 0$, $\vert \tilde h_{\eta}f_i\vert$ and $\vert D\tilde h_{\eta}f_i\vert$ are uniformly bounded, moreover $\Delta \tilde h_{\eta}f_i\in L^{\infty}(m)$ (see for instance (3.2.3) in \cite{Gnon}). Hence for sufficiently small $\eta>0$, $\tilde f_i:=\tilde h_{\eta}f_i$ satisfies 
%\begin{align}
% \dashint_{B_R(\bar x)}\vert \Gamma(\tilde f_i,\tilde f_j)-\delta_{i,j}\vert^2\,dm\leq C\epsilon.\notag
%\end{align}
Let $A$ be a subset defined as 
\begin{align}
 A:=\left\{y\in B_{R/10^2}(\bar x)\;;\;\dashint_{B_s(y)}\left\vert \Gamma(g^i,g^j)-\delta_{ij}\right\vert\,dm\leq\Psi(\sqrt{\epsilon};K,N)\quad\text{for any $s<R/10^2$}\right\}.\label{eq:effective}
\end{align} 
By Lemma \ref{lem:telescope}, $m(A)/m(B_{R/10^2}(\bar x))\geq 1-\Psi(\sqrt{\epsilon};K,N)$, in particular $m(A)>0$. Since $m(X\setminus \calR)=0$, we have $m(A\cap \calR)>0$. Take a point $y\in A\cap \calR$ and assume $y\in \calR_l$ for some $1\leq l\leq N$. Without loss of generality, $B_s(y)\subset B_R(\bar x)$. Note that, for such $y$, we have 
\begin{align}
 \dashint_{B^{d_s}_1(y)}\vert \Gamma(g^i,g^j)-\delta_{ij}\vert\,dm^y_s\leq\Psi(\sqrt{\epsilon};K,N).\notag
\end{align} 
Since $f^i$ is a Lipschitz function, $g^i\in\TestF(X)$ and $\Delta g^i\in L^{\infty}(m)$.  
Let us consider the effect of rescaling of metric to the Cheeger energy, Laplacian, and the Gamma operator. 
%%%%%%%%%%%%
\begin{lem}\label{lem:rescaleeffect}
 Denote the Cheeger energy, Laplacian, heat semigroup, Gamma operator in the rescaled space $(X,d_s,m)$ by $\Ch^{d_s}$, $\Delta^{d_s}$, $\{h_t^{d_s}\}_{t>0}$, and $\Gamma^{d_s}$ respectively. Then 
 \begin{align}
  \Ch^{d_s}=s^2\Ch,\quad \Delta^{d_s}h^{d_s}_t=s^2\Delta h_{s^2t}, \quad \Gamma^{d_s}=s^2\Gamma.\notag
 \end{align}
\end{lem}
%%%%%%%%%%%%
\begin{proof}
 Since 
 \begin{align}
  \frac{\vert f(x)-f(y)\vert}{d_s(x,y)}=s\frac{\vert f(x)-f(y)\vert}{d(x,y)}\notag
 \end{align}
 holds, it is easy to get $\vert \nabla^{d_s}f\vert=s\vert \nabla f\vert$. By the definition of $\Ch$, we have $\Ch^{d_s}=s^2\Ch$. A similar calculation let us obtain $\Gamma^{d_s}=s^2\Gamma$. For an appropriate function $f$, consider the heat equation $\partial_th_t^{d_s}f=\Delta^{d_s}h^{d_s}_tf$. Since $h_{s^2t}f$ satisfies 
 \begin{align}
  \partial_th_{s^2t}f=s^2\Delta h_{s^2t}f=\Delta^{d_s}h_{s^2t}f\notag
 \end{align}
 and $\lim_{t\rightarrow+0}h_{s^2t}f=f$ in $L^2$ sense. By the uniqueness of the heat flow, we obtain $h^{d_s}_tf=h_{s^2t}f$. Hence $\Delta^{d_s}h^{d_s}_t=s^2\Delta h_{s^2t}$ holds for any $t>0$. 
\end{proof}
%%%%%%%%%%%%
Hence 
\begin{align}
 \dashint_{B^{d_s}_1(y)}\vert \Gamma^{d_s}(s^{-1}g^i,s^{-1}g^j)-\delta_{ij}\vert\,dm^y_s\leq \Psi(\sqrt{\epsilon};K,N)\notag
\end{align}
holds. Define new functions $k^i_s:=(g^i_s-g^i_s(y))\varphi_1^{d_s,y}:=s^{-1}(g^i-g^i(y))\varphi^{d_s,y}_1$, $i=1,\ldots,k$. 
%%%%%%%%%%%%
\begin{prop}\label{prop:kfunctions}
 There exists a limit function $K^i$ on $\R^l\in \Tan(X,y)$ for each $i$ such that $K^i\in\calD(\Delta)$ and $k^i_s$ converges to $K^i$ in $W^{1,2}$-strongly and $\Delta^{d_s}k^i_s$ converges to $\Delta K^i$ in $L^2$-weakly. 
 %Moreover $K^i$, $i=1,\ldots,k$, are harmonic functions in $\R^l$. 
\end{prop}
%%%%%%%%%%%%
\begin{proof}
\par \underline{Step 1: The functions $k^i_s$ belong $W^{1,2}(m^y_s)$. } Let us calculate the $L^2(m^y_s)$-norm of $k^i_s$. We denote by $D$ the normalized constant so that $Dm=m^y_s$ for abbreviation. Since the heat flow preserves the $L^2$-norm of a function, note that $\Vert \hat h_t f\Vert_{L^2(m)}\leq \Vert f\Vert_{L^2(m)}$.   
\begin{align}
 \int_X\vert k^i_s\vert^2\,dm^y_s&=D\int_{B^{d_s}_2(y)}\vert (g^i_s-g^i_s(y))\varphi^{d_s,y}_1\vert^2\,dm\label{eq:lem:indepharm}\\
 &\leq D\int_{B^{d_s}_2(y)}\vert \hat{h}_t(s^{-1}f_i-s^{-1}f_i(y))\vert^2\,dm\notag\\
 &\leq D\int_{B^{d_s}_2(y)}\vert (d_s^{p_i}-d_s^{p_i}(y))\vert^2\,dm\notag\\
 &\leq 4m^y_s(B^{d_s}_2(y))<\infty.\notag
\end{align}
Also by the locality property of Dirichlet form $\Ch$, we obtain 
\begin{align}
 \int_X\vert\nabla^{d_s}k^i_s\vert^2\,dm^y_s
 &=\int_X\Gamma^{d_s}(k^i_s,k^i_s)\,dm^y_s\notag\\
 &=\int_{B^{d_s}_2(y)}\Gamma^{d_s}(k^i_s,k^i_s)\,dm^y_s\notag.
\end{align}
By the Leibniz rule and the locality property, we have $\vert\nabla^{d_s}k^i_s\vert\leq \vert(g^i_s-g^i_s(y))\vert\vert\nabla^{d_s} \varphi^{d_s,y}_1\vert+\varphi^{d_s,y}_1\vert \nabla^{d_s} g^i_s\vert$. Hence 
\begin{align}
 &\int_{B^{d_s}_2(y)}\Gamma^{d_s}(k^i_s,k^i_s)\,dm^y_s\leq 2\int_{B^{d_s}_2(y)}\vert g^i_s-g^i_s(y)\vert^2\vert\nabla^{d_s}\varphi^{d_s,y}_1\vert^2+(\varphi^{d_s,y}_1)^2\vert \nabla^{d_s}g^i_s\vert^2\,dm^y_s\notag\\
 &\leq 8C^2m^y_s(B^{d_s}_2(y))+2\int_{B^{d_s}_2(y)}\Gamma^{d_s}(g^i_s,g^i_s)\,dm^y_s\notag\\
 &=8C^2m^y_s(B^{d_s}_2(y))+2D\int_{B^{d_s}_2(y)}\Gamma(g^i,g^i)\,dm\notag\\
 &\leq 8C^2m^y_s(B^{d_s}_2(y))+2De^{K_-t}\int_{B^{d_s}_2(y)}\Gamma(f^i,f^i)\,dm\notag\\
 &\leq 8C^2m^y_s(B^{d_s}_2(y))+2e^{K_-t}m^y_s(B^{d_s}_2(y))<\infty\notag
\end{align}
holds. 
\smallskip
\par \underline{Step 2: The existence of limit functions.} By Step 1, we obtain 
\begin{align}
 &\Vert k^i_s\Vert^2_{L^2(m^y_s)}\leq 4m^y_s(B^{d_s}_2(y))\notag\\
 &2\Ch^{d_s}(k^i_s)\leq (8C^2+2e^{K_-t})m^y_s(B^{d_s}_2(y)).\notag
\end{align}
Since $y\in \calR_l$, $\limsup_{s\rightarrow 0}m^y_s(B^{d_s}_2(y))<\infty$. Therefore 
\begin{align}
 \sup_{s\in (0,1)}\left(\Vert k^i_s\Vert^2_{L^2(m^y_s)}+\Ch^{d_s}(k^i_s)\right)\leq (4C^2+e^{K_-t}+4)\sup_{s\in(0,1)}m^y_s(B^{d_s}_2(y))<\infty.\notag 
\end{align}
Moreover by (\ref{eq:lem:indepharm}), 
\begin{align}
 \int_{X\setminus B^{d_s}_R(y)}\vert k^i_s\vert^2\,dm^y_s=0\notag
\end{align}
whenever $R\geq 3$. Thus applying Theorem \ref{thm:precptSobolev} leads the existence of a convergent subsequence of $\{k^i_s\}_{s\in (0,1)}$ and the limit function $K^i\in W^{1,2}(\R^l)$ for each $i$ such that $k^i_s$ $L^2$-strongly converges to $K^i$. 
\end{proof}
%%%%%%%%%%%%
%%%%%%%%%%%%
\begin{prop}\label{prop:lk}
 $l\geq k$. 
\end{prop}
%%%%%%%%%%%%
\begin{proof}
 Recall that we have limit functions $K^i$, $i=1,\ldots,k$ in $\R^l$. By Theorem \ref{thm:convinnerproduct}, we have 
 \begin{align}
  &\dashint_{B_1(0)}\vert\Gamma(K^i,K^i)-1\vert\,d\underline{\calL^l}\leq \Psi(\sqrt{\epsilon};K,N)\label{eq:almost1}\\
  &\dashint_{B_1(0)}\vert\Gamma(K^i,K^j)\vert\,d\underline{\calL^l}\leq \Psi(\sqrt{\epsilon};K,N)\;\quad\text{for }i\neq j.\label{eq:almostperpendicular}
 \end{align}
 (\ref{eq:almost1}) implies the non-triviality of each $K^i$. Again applying Lemma \ref{lem:telescope} guarantees the existence a measurable subset $B\subset B_1(0)$ of positive measure such that for any $p\in B$, 
 \begin{align}
  \dashint_{B_s(p)}\vert \Gamma(K^i,K^j)\vert\,d\underline{\calL^l}\leq \Psi(\epsilon^{1/4};K,N)\notag
 \end{align} 
 holds for any $s\in(0,1)$. Since $\calL^l(B)>0$, at almost every point $p\in B$, we obtain 
 \begin{align}
  \Gamma(K^i,K^j)(p)=\lim_{s\rightarrow 0}\dashint_{B_s(p)}\Gamma(K^i,K^j)\,d\underline{\calL^l}\leq \Psi(\epsilon^{1/4};K,N)\notag
 \end{align}
 by Lebesgue's differentiation theorem. Hence $\{\nabla K^i(p)\}_{i=1}^k$ is linearly independent in $T_p\R^l\cong \R^l$. Therefore $l$ should be at least $k$. 
\end{proof}
The following corollary is easy to prove. 
%%%%%%%%%%%%
\begin{cor}
 Let $(X,d,m)$ be an $\RCD^*(K,N)$ space for $K\in\R$, $N\in\bbN$. Assume $\calR_N\neq \emptyset$. Then $\mathrm{dim}_H (X,d)=N$. 
\end{cor}
%%%%%%%%%%%%
\begin{proof}
 By Propositions \ref{prop:kfunctions}, \ref{prop:lk}, we have $m(\calR_N)>0$. In \cite{GP,KeM,DPMR}, they prove that on $\calR_N$, $m\sim\calH^N$. Thus $\mathrm{dim}_H(X,d)\geq N$. On the other hand, Sturm proved $\mathrm{dim}_H(X,d)\leq N$ in \cite{Stmms2}. We have the conclusion.  
\end{proof}
\section{Dimension on $\RCD$ spaces}
 We define a version of dimension of an $\RCD$ space. 
 %%%%%%%%%%
 \begin{defn}\label{def:dimmm}
  Let $\calX:=(X,d,m)$ be an $\RCD^*(K,N)$ space for $K\in\R$ and $N\in(1,\infty)$. The dimension of $(X,d,m)$ is the largest number $k$ so that $\calR_k\neq\emptyset$ and $\calR_l=\emptyset$ for any $l>k$. For simplicity, $\Dim \calX$ denotes the dimension of $(X,d,m)$. 
 \end{defn}
 %%%%%%%%%%
 %%%%%%%%%%
 \begin{rem}\label{rem:positivity}
  By Theorem \ref{thm:posi}, $\Dim \calX=k$ implies $m(\calR_k)>0$. In general, $\Dim \calX\leq \mathrm{dim}_H(X,d)$. Even in the case of collapsing Ricci limit spaces, we do not know the equality holds or not. 
 \end{rem}
 %%%%%%%%%%
 \begin{rem}
  Let $(X,d_X,m_X)$ and $(Y,d_Y,m_Y)$ be $\RCD^*(K,N)$ spaces for $K\in\R$, $N\in(1,\infty)$. Assume that $(X,d_X,m_X)$ and $(Y,d_Y,m_Y)$ are isomorphic to each other. Then by the definition, $\Dim (X,d_X,m_X)=\Dim (Y,d_Y,m_Y)$ holds. Moreover it follows that $\Dim (X,d,m_1)=\Dim (X,d,m_2)$ if $(X,d,m_1)$ and $(X,d,m_2)$ are $\RCD^*(K,N)$ spaces and $\supp m_1=\supp m_2=X$. Indeed, it suffices to prove that $\Dim \calX_1\leq \Dim \calX_2$ by the symmetry of the condition, where $\calX_i=(X,d,m_i)$. Put $k=\Dim \calX_1$ and denote by $\calR^i_l$ the $l$-dimensional regular set with respect to $\calX_i$. Take a point $x\in\calR^1_k$ and fix it. Since $(X,d_r,(m_1)^x_r,x)\rightarrow (\R^k,d_{\R^k},\underline{\calL}^k,0)$ in the measured Gromov-Hausdorff sense implies $(X,d_r,x)\rightarrow (\calR^k,d_{\R^k},0)$ in the Gromov-Hausdorff one, we have that the set of metric tangent cones at $x$ is unique and isometric to $(\R^k,d_{\R^k},0)$. Hence $\Tan(X,m_2,x)=\{(\R^k,d_{\R^k},\nu,0)\;;\;\nu\text{ is a limit measure on }\R^k\}$, where a limit measure means $(m_2)^x_{r_i}\rightarrow \nu$, for a non-increasing sequence $\{r_i\}$ tending to 0, in the sense of pointed measured Gromov convergence. It remains to prove that $\nu$ must be $\underline{\calL}^k$, that is, $x\in\calR^2_k$. However, $(\R^k,d_{\R^k},\nu)$ is an $\RCD^*(0,N)$ space with $k$-th straight lines. Thus the splitting theorem \cite{Gsplit} tells us that $\nu=\underline{\calL}^k$. This completes the proof. 
 \end{rem}
 %%%%%%%%%%
 Before stating the main result in this section, we give the following definitions and results developed by Ambrosio and Honda \cite{AHlocal}. For an open set $A\subset X$, $\mathrm{LIP}_c(A,d)$ denotes the set of all Lipschitz functions whose supports are compact and contained in $A$. 
 %%%%%%%%%%
 \begin{defn}
  Let $U$ be an open subset in $X$. 
  \begin{enumerate}
   \item Define $W^{1,2}_0(U,d,m):=\overline{\mathrm{Lip}_c(U,d)}^{\Vert\cdot\Vert_{W^{1,2}}}$.
   \item Define $\hat{W}_0^{1,2}(U,d,m):=\{f\in W^{1,2}\;;\;f=0\;m\text{-a.e. on }X\setminus U\}$.
  \end{enumerate}
 \end{defn}
 %%%%%%%%%%
  It follows that $W^{1,2}_0(B_R(x),d,m)=\hat W_0^{1,2}(B_R(x),d,m)$ for any $x\in X$ and for every $R>0$ excepting at most countably many positive numbers (see \cite{AHlocal}*{Lemma 2.12}). 
 %%%%%%%%%%
  Analogously \emph{the local Cheeger energy} is defined as 
  \begin{align}
   \Ch_{(x,R)}(f):=\begin{cases}
                            \Ch(f)&\text{if }f\in W^{1,2}_0(B_R(x),d,m)\\
                            +\infty&\text{otherwise}.
                           \end{cases}\notag
  \end{align}
 %%%%%%%%%%
 Assume $(X,d_n,m_n,x_n)\rightarrow(X,d_{\infty},m_{\infty},x_{\infty})$ and $\supp m_n\ni z_n\rightarrow z\in\supp m_{\infty}$ in the pmG sense. The \emph{Mosco convergence} for local Cheeger energies, denoted by 
 \begin{align}
  \Ch^n_{\mathsf{loc}}\rightarrow \Ch^{\infty}_{\mathsf{loc}}\quad\text{at }(z,R)\in X\times R,\notag
 \end{align}
 is defined as follows; 
 \begin{enumerate}
  \item For every $f_n\in L^2(B_{R_n}(z_n),m_n)$ $L^2$-weakly converging to $f\in L^2(B_R(z),m_{\infty})$, 
  \begin{align}
   \Ch^{\infty}_{(z,R)}(f)\leq \liminf_{n\rightarrow\infty}\Ch^n_{(z_n,R_n)}(f_n)\notag
  \end{align}
  holds, where $z_n\rightarrow z$ and $R_n\rightarrow R$. 
  \item For every $f\in L^2(B_R(z),m_{\infty})$ there exist $f_n\in L^2(B_{R_n}(z_n),m_n)$ $L^2$-strongly converging to $f$ such that 
  \begin{align}
   \Ch^{\infty}_{(z,R)}(f)=\lim_{n\rightarrow\infty}\Ch^n_{(z_n,R_n)}(f_n)\notag
  \end{align}
  holds, where $R_n$ and $z_n$ are the same as above. 
 \end{enumerate}
 %%%%%%%%%%
 \begin{thm}[\cite{AHlocal}*{Theorem 3.4}]\label{thm:localMosco}
  The following are equivalent; 
  \begin{enumerate}
   \item $\Ch^n_{\mathsf{loc}}\rightarrow\Ch^{\infty}_{\mathsf{loc}}$ at $(z,R)$.
   \item $W^{1,2}_0(B_R(z),d,m)=\hat W_0^{1,2}(B_R(z),d_{\infty},m_{\infty})$.
  \end{enumerate}
 \end{thm}
 %%%%%%%%%%
 We define the new Sobolev space as follows. 
 %%%%%%%%%%
 \begin{defn}
  $W^{1,2}(B_R(z),d,m)$ consists of $f\in L^2(B_R(z),m)$ those which satisfy the following two conditions; 
  \begin{enumerate}
   \item $\phi f\in W^{1,2}(X,d,m)$ for any $\phi\in\mathrm{LIP}_c(B_R(z))$,
   \item $\Gamma(f)\in L^1(B_R(z))$.
  \end{enumerate}
 \end{defn}
 %%%%%%%%%%
 We also define the local $W^{1,2}$-convergence. 
 %%%%%%%%%%
 \begin{defn}
  Let $f_n\in W^{1,2}(B_R(z_n),d_n,m_n)$ and $f\in W^{1,2}(B_R(z),d_{\infty},m_{\infty})$. We say that $f_n$ \emph{$W^{1,2}$-weakly converges} to $f$ on $B_R(z)$ if $f_n$ $L^2$-weakly converges to $f$ on $B_R(z)$ with $\sup_n\Vert f_n\Vert_{W^{1,2}}<\infty$. Furthermore, the \emph{$W^{1,2}$-strongly convergence} on $B_R(z)$ is defined by requiring the $W^{1,2}$-weakly convergence on $B_R(z)$ and $\lim_n\Vert\Gamma_n(f_n)\Vert_{L^1(B_R(z_n))}=\Vert\Gamma(f)\Vert_{L^1(B_R(z))}$. 
 \end{defn}
 %%%%%%%%%%
 The following Theorem plays a key role in the main result in this section. 
 %%%%%%%%%%
  \begin{thm}[\cite{AHlocal}*{Corollary 4.3}]\label{thm:localsum}
   Let $f_n,g_n\in W^{1,2}(B_R(z_n),d_n,m_n)$ be $W^{1,2}$-strongly convergent sequences to $f,g\in W^{1,2}(B_R(z),d_{\infty},m_{\infty})$ on $B_R(z)$ respectively. Then 
   \begin{align}
    \lim_{n\rightarrow\infty}\int_{B_R(z_n)}\Gamma^n(f_n,g_n)\,dm_n=\int_{B_R(z)}\Gamma^{\infty}(f,g)\,dm_{\infty}.\notag
   \end{align} 
  \end{thm}
 %%%%%%%%%%
 \begin{rem}\label{rem:sumLp}
  For two sequences $f_i,g_i$ $L^p$-strongly converging to $f,g$ respectively, the sum of these functions $f_i+g_i$ also $L^p$-strongly converges to $f+g$ for $p\in[1,\infty)$. See \cite{HRicLp} for the proof. In \cite{HRicLp}, the definition of $L^p$-convergence looks different from that in this paper, but it is actually equivalent. The proof of equivalence is also found in the same article. 
 \end{rem}

 %%%%%%%%%%
  %\begin{thm}[Mosco convergence of Cheeger energies, \cite{GMS}*{Theorem 6.8}]\label{thm:Mosco}
  %~\\
  %Let $(X,d,m_n,x_n)$ be $\RCD^*(K,N)$ spaces for $K\in \R$ and $N\in(1,\infty)$, $n\in\bbN\cup\{\infty\}$ and $\Ch_n$ Cheeger energies in $L^2(X,m_n)$. Assume $(X,d,m_n,x_n)$ converges to $(X,d,m_{\infty},x_{\infty})$ in the pmG sense. Then the following hold; 
  % \begin{enumerate}
  %  \item For every sequence $n\mapsto f_n\in L^2(X,m_n)$ that $L^2$-weakly converges to some $f_{\infty}\in L^2(X,m_{\infty})$, we have 
  %  \begin{align}
 %    \Ch_{\infty}(f_{\infty})\leq \liminf_{n\rightarrow\infty}\Ch_n(f_n).\notag
 %   \end{align}
 %   \item For every $f_{\infty}\in L^2(X,m_{\infty})$, there exists a sequence $n\mapsto f_n\in L^2(X,m_n)$ that $L^2$-strongly converges to $f_{\infty}$ such that 
 %   \begin{align}
 %    \Ch(f_{\infty})=\lim_{n\rightarrow\infty}\Ch_n(f_n).\label{eq:gammastrong}
 %   \end{align}
 %  \end{enumerate}
 % \end{thm}
 %%%%%%%%%%
 Combining Theorem \ref{thm:localMosco} and Theorem \ref{thm:convinnerproduct} leads the following. 
 %%%%%%%%%%
  \begin{thm}\label{thm:lscdimmm}
   Let $\calX_n:=(X,d,m_n)$, $x_n\in\supp m_n$ $n\in\bbN\cup\{\infty\}$ be $\RCD^*(K,N)$ spaces for $K\in\R$ and $N\in(1,\infty)$ and assume $(\calX_n,x_n)$ converges to $(\calX_{\infty},x_{\infty})$ in the pmG sense. Then 
   \begin{align}
    \Dim \calX_{\infty}\leq\liminf_{n\rightarrow\infty}\Dim\calX_n.\notag
   \end{align}
  \end{thm}
 %%%%%%%%%%
 \begin{proof}
  Let $k$ be the dimension of $\calX_{\infty}$. By Remark \ref{rem:positivity}, we have $m_{\infty}(\calR_k)>0$. Then for any $\epsilon>0$, there exist a subset $A\subset\calR_k$ with $m_{\infty}(A)>0$ and a family of functions $\{g^i\}_{i=1}^k$ such that 
  \begin{align}
   \dashint_{B_s(y)}\vert\Gamma^{\infty}(g^i,g^j)-\delta_{ij}\vert\,d(m_{\infty})^y_s<\epsilon\notag
  \end{align}
  holds for any $s\in(0,1)$ and $y\in A$, after rescaling the metric and the measure if needed. By Theorem \ref{thm:localMosco}, there exist a positive number $s\in(0,1)$, functions $g^i_n\in L^2(B_{s_n}(y_n))$ $L^2$-strongly converging to $g^i$ such that 
  \begin{align}
   \Ch_{(y,s)}(g^i)=\lim_{n\rightarrow \infty}\Ch_{(y_n,s_n)}(g^i_n)\notag
  \end{align} 
  holds for $i=1,\ldots,k$. Thus we have 
  \begin{align}
   \lim_{n\rightarrow\infty}\dashint_{B_s(y_n)}\Gamma^n(g^i_n,g^j_n)\,dm_n=\dashint_{B_s(y)}\Gamma^{\infty}(g^i,g^j)\,dm_{\infty}<\epsilon\label{eq:small1}
  \end{align}
  for $i\neq j$ by Theorem \ref{thm:localsum}. Note that by the Bishop-Gromov inequality, $m_n(B)\rightarrow m_{\infty}(B)$ for any metric ball $B$. 
  Since it is also clear that $\chi_{B_1(y)}$ $W^{1,2}$-strongly converges to itself, 
  \begin{align}
   \lim_{n\rightarrow\infty}\dashint_{B_{s}(y_n)}\vert \Gamma^n(g^i_n,g^j_n)-1\vert\,dm_n=\dashint_{B_s(y)}\vert\Gamma^{\infty}(g^i,g^j)-1\vert\,dm_{\infty}\label{eq:small2}
  \end{align} 
  holds by Remark \ref{rem:sumLp}. The same argument as in Section \ref{sec:main} leads $m_n(\calR_k)>0$. Therefore we have $\Dim \calX_n\geq k$ for sufficiently large $n$. This completes the proof. 
%  By (\ref{eq:gammastrong}) in Theorem \ref{thm:Mosco}, there exist functions $g^i_n\in L^2(X,m_n)$ $L^2$-strongly convergent to $g^i$ such that 
%  \begin{align}
%   \Ch_{\infty}(g^i)=\lim_{n\rightarrow \infty}\Ch_n(g^i_n).\notag
%  \end{align}
%  Since $g^i_n$ $W^{1,2}$-strongly converges to $g^i$ for every $i=1,\ldots,k$, we have $\Gamma^{n}(g^i_n,g^j_n)\rightarrow \Gamma^{\infty}(g^i,g^j)$ in $L^1$-strongly sense by applying Theorem \ref{thm:convinnerproduct}. Thus for fixed $s\in(0,1)$, since $m_i$ and $m_{\infty}$ are doubling measures, we obtain 
%  \begin{align}
%   \dashint_{B_s(y_n)}\vert\Gamma^n(g^i_n,g^j_n)-\delta_{ij}\vert\,d(m_n)^{y_n}_s\leq 2\epsilon\notag
%  \end{align}
%  for some $y_n\in\supp m_n$ converging to $y$. The same argument as in Section \ref{sec:main} leads $m_n(\calR_k)>0$. Therefore we have $\Dim \calX_n\geq k$. This completes the proof. 
 \end{proof}
 %%%%%%%%%%
  The following is an easy consequence from Theorem \ref{thm:lscdimmm}
 %%%%%%%%%%
 \begin{cor}
  Let $(X,d,m)$ be an $\RCD^*(K,N)$ space of dimension $k$. And define a set $\WE_i$ as 
  \begin{align}
   \WE_i:=\left\{x\in X\;;\;\R^i\times Z\in\Tan(X,x)\text{ with }\di\, Z>0\right\}.\notag
  \end{align}
  Then $\WE_k=\emptyset$. 
 \end{cor}
 %%%%%%%%%%
 \begin{rem}\label{rem:analydim}
 Recently, the study of a version of differential structure on metric measure spaces is developed by Gigli \cite{Gnon}. Roughly speaking, the $L^2$-sections of cotangent bundles over a metric measure space are considered. Hence the dimension of such spaces makes sense. Indeed, Han define the \emph{analytic dimension} on a metric measure space with such a differential structure \cite{HanRictensor}. Combining our result and Theorem 3.3 in \cite{GPequiv} leads the equivalence of analytic dimension and our dimension defined in this section. 
 \end{rem}
 %%%%%%%%%%
 \begin{rem}
  After finishing this paper, the author find the paper by De Philippis and Gigli on arXiv \cite{DPG}. They introduce a notion of \emph{non-collapsed $\RCD$ spaces} and study that. Moreover they study much about the Hausdorff dimension not only non-collapsed $\RCD$ spaces but also the usual ones. They prove that $\dim_H(X,d)\leq N-1$ for $\RCD(K,N)$ spaces with $N\in\bbN$ unless $\dim_H(X,d)=N$. Combining their result with ours leads the following; Let $(X,d,m)$ be an $\RCD^*(K,N)$ space for $N\in\bbN$. Then $\dim (X,d,m)=N-1$ if and only if $\dim_H(X,d)=N-1$. 
 \end{rem}

\section*{Acknowledgement}
The author would like to thank Professor Shouhei Honda for their helpful comments and fruitful discussions. He is supported by Grant-in-Aid for Young Scientists (B) 15K17541. 

\end{document}